\documentclass[12pt,twoside,reqno]{amsart}
\usepackage{amsmath, amssymb, amsthm, color}
\usepackage[colorlinks=false,citecolor=blue]{hyperref}
\usepackage[tmargin=1in,bmargin=1in,rmargin=1in,lmargin=1in]{geometry}

\newtheorem{theorem}{Theorem}[section]
\newtheorem{lemma}[theorem]{Lemma}
\newtheorem{prop}[theorem]{Proposition}
\newtheorem{cor}[theorem]{Corollary}
\theoremstyle{definition}

\theoremstyle{remark}
\newtheorem{remark}[theorem]{\bf{Remark}}
\numberwithin{equation}{section}
\allowdisplaybreaks
\begin{document}

\title[ Refined numerical radius estimates and Euclidean operator radius]
{ Refined numerical radius estimates and Euclidean operator radius}

\author{Pintu Bhunia and Rukaya Majeed}
\address[Bhunia]{Department of Mathematics, SRM University AP, Amaravati 522240, Andhra Pradesh, India}
\email{\tt pintubhunia5206@gmail.com ; pintu.b@srmap.edu.in}

\address[Majeed]{Department of Mathematics, SRM University AP, Amaravati 522240, Andhra Pradesh, India} 
\email{rukayamajeed20@gmail.com ; rukaya\_majeed@srmap.edu.in} 




\date{\today}
\maketitle

\begin{abstract}
   We obtain new lower and upper bounds for the numerical radius of a bounded linear operator $A$ on a complex Hilbert space, which refine the existing ones. 
   In particular, if $w(A)$ and $\|A\|$ denote the numerical radius and operator norm of $A$, respectively, then we show that 
     \begin{eqnarray*}
       \nu(A) + \frac{1}{4} \left\||A|^2+|A^*|^2\right\| 
         \leq w^2(A) \leq  \frac12 w\left(\frac{|A|+|A^*|}{2}A \right)+ \frac14 \left\| |A|^2+ \left( \frac{|A|+|A^*|}{2}\right)^2 \right\|, 
    \end{eqnarray*}
    where $\nu(A)\geq 0$ is a real number involving the operator norm of the Cartesian decomposition of $A$. We also develop several new numerical radius inequalities for the products and sums of operators via Euclidean operator radius of $2$-tuples of operators.
    In addition, we deduce equality characterizations for the inequalities.
    As an application, we obtain numerical radius inequalities for the commutators of operators, which improves the Fong and Holbrook's inequality $w(AB\pm BA) \leq 2\sqrt{2} w(A) \|B\|$ [Canadian J. Math. 1983].
    \end{abstract}


\section{Introduction}

Suppose $\mathcal{B}(\mathcal{H})$ denotes the $C^*$-algebra of all bounded linear operators on a complex Hilbert space $\mathcal{H}$.
The Cartesian decomposition of $A\in \mathcal{B}(\mathcal{H})$ is defined as $A=Re(A)+ i Im(A),$ where $Re(A)=\frac{1}{2}(A+A^*)$ and $Im(A)=\frac{1}{2i}(A-A^*)$. We denote $|A|=\sqrt{A^*A}$, where $A^*$ is the adjoint of $A$.
 The operator norm and numerical radius of $A\in\mathcal{B}(\mathcal{H})$ are defined, respectively, as
\[\|A\|=\sup \{ \|Ax\|: x\in \mathcal{H}, \|x\|=1\} {\text{ and }}  w(A)=\sup \{ \left|\left<Ax,x\right>\right| : x\in \mathcal{H}, \|x\|=1\}.\]
The search of possible relations between these two quantities, namely $\|A\|$ and $w(A)$, has acquired the interest of numerous researchers due to the simplicity we find in computing $\|A\|$, and to the possible observations that may arise from such relations.
As a basic relation in this direction, we have \cite[Theorem 1.3-1]{Gustafson_Book_1997}:
\begin{equation}\label{Eq_Intro_Equiv}
\frac{1}{2}\|A\|\leq w(A)\leq \|A\|,
\end{equation}
\text{for all} $A\in\mathcal{B}(\mathcal{H}).$
These bounds are sharp, $\frac{1}{2}\|A\|=w(A)$ if $A^2=0$ and $ w(A)= \|A\|$ if $A$ is normal.
These classical bounds have been refined by several researchers. 
Here, first we mention the following well known bounds (see \cite{Kittaneh1, Kittaneh2, DRAG}):
\begin{equation}\label{Kittaneh_Eq_2}
w(A) \leq \frac{ 1}{2} \left\|     |A|+|A^*| \right\| ,
\end{equation}

\begin{equation}\label{Kittaneh_Eq_1}
{\frac{1}{4}\left\| |A|^2+|A^*|^2 \right|} \leq w^2(A)\leq {\frac{1}{2}\left\| |A|^2+|A^*|^2\right\|}
\end{equation}
and
\begin{eqnarray}\label{EQNDRA}
    w^2(A) \leq {\frac12w(A^2)+\frac12 \|A\|^2}.
\end{eqnarray}
Further, Abu-Omar and Kittaneh \cite{Abu-Omar-Kittaneh} (see also \cite{Bhunia-LAA}) developed the upper bound
\begin{eqnarray}\label{A1}
    w^2(A) &\leq& \frac12 w(A^2)+\frac14 \| |A|^2+|A^*|^2\|.
\end{eqnarray}
Recently, Bhunia and Paul \cite{Bhunia-BSM} also established the upper bound 
\begin{eqnarray}\label{B1}
    w^2(A) &\leq& \frac12 w(|A||A^*|)+\frac14 \| |A|^2+|A^*|^2\|.
\end{eqnarray}
The above two bounds \eqref{A1} and \eqref{B1} are incomparable, in general. These both bounds refine the upper bound in \eqref{Kittaneh_Eq_1}. 
Further refinements and discussion on numerical radius and operator norm relations can be found in a long list of references, of which we refer \cite{Feki, BP, BHU, BKS, BHU-G, KMS, Kit23, SMS, SM}. We also refer the reader to the books \cite{Bhunia-book, Gustafson_Book_1997, Wu} for more systematic studies on this topic.

Our main aim in this work is to study further inequalities and relations involving the operator norm, numerical radius and Euclidean operator radius (an extension of the classical numerical radius). In particular, we obtain several improved lower and upper bounds for the numerical radius.

The definition of the numerical radius of a bounded linear operator suggests that the most straightforward approach to studying numerical radius inequalities is to investigate inner product inequalities. In this regard, the Cauchy-Schwarz inequality is considered a cornerstone, where for $x,y\in\mathcal{H}$, we have	
\begin{equation*}
\left\vert \left\langle x,y\right\rangle \right\vert \leq \left\Vert
x\right\Vert \left\Vert y\right\Vert .  \label{moo}
\end{equation*}
A strengthened version of the Cauchy-Schwarz inequality is known as the Buzano inequality, where we have the following lemma.

\begin{lemma} \cite{Buzano}\label{lem1} 
    If $x,y,e\in \mathcal{H}$ with $\|e\|=1,$ then
    $$|\langle x,e\rangle \langle e,y\rangle| \leq \frac12 \left(\|x\| \|y\|+ |\langle x,y\rangle|   \right) .$$
\end{lemma}

Another useful generalization of the Cauchy-Schwarz inequality is the mixed Schwarz inequality, which states that 
\begin{equation*}
\left\vert \left\langle Ax,y\right\rangle \right\vert ^{2}\leq \left\langle
\left\vert A\right\vert x,x\right\rangle \left\langle \left\vert A^{\ast
}\right\vert y,y\right\rangle ,  \label{moo2}
\end{equation*}
for every $x,y\in \mathcal{H}$.  One more generalization in this direction is the following lemma.

\begin{lemma}\label{lem2} \cite[Theorem 5]{Kittaneh88}
    Let $A,B\in \mathcal{B}(\mathcal{H})$ and $x,y\in \mathcal{H}.$ If $f,g$ are two continuous nonnegative functions on $[0,\infty)$ such that $f(t)g(t)=t$, for all $t\geq 0,$ then
      \[|\langle ABx,y\rangle|\leq r(B) \|f(|A|)x\|\|g(|A^*|)y\| \quad \text{for all } x, y\in \mathcal{H} \]
     where $r(B)$ denotes the spectral radius. In particular, for $B=I$,
    $$|\langle Ax,y\rangle|\leq \| f(|A|)x\|  \| g(|A^*|)y\|.$$
\end{lemma}

In this paper, using these inner product inequalities, we establish several new numerical radius inequalities, which refine the previously related inequalities. 
Further, we develop several numerical radius inequalities via Euclidean operator radius and Euclidean operator norm of $2$-tuples of operators.
These also provides new refinements of the existing bounds.
We study equality conditions of the numerical radius inequalities. In addition, we obtain numerical radius inequalities for the commutators of operators, which refine and generalize the well known inequality $w(AB\pm BA) \leq 2\sqrt{2} w(A) \|B\|,$ given by Fong and Holbrook \cite{Fong}.


\section{Upper bounds for the numerical radius}
We begin this section with presenting the following inequalities involving two non-negative continuous functions, from which we derive several improved upper bounds for the numerical radius of bounded linear operators.


\begin{theorem}\label{th2}
     Let $A\in \mathcal{B}(\mathcal{H})$. If $f,g$ are two continuous nonnegative functions on $[0,\infty)$ such that $f(t)g(t)=t$, for all $t\geq 0,$ then
     \begin{eqnarray*}
         w^2(A) \leq \frac12 w\left(A \frac{f^2(|A|)+g^2(|A^*|)}{2} \right)+ \frac14 \left\| |A^*|^2+ \left(\frac{f^2(|A|)+g^2(|A^*|)}{2} \right)^2 \right\|
     \end{eqnarray*}
     and 
      \begin{eqnarray*}
         w^2(A) \leq \frac12 w\left( \frac{f^2(|A^*|)+g^2(|A|)}{2}A \right)+ \frac14 \left\| |A|^2+ \left(\frac{f^2(|A^*|)+g^2(|A|)}{2} \right)^2 \right\|.
     \end{eqnarray*}
\end{theorem}

\begin{proof}
    Take $x\in \mathcal{H}$ with $\|x\|=1.$ From Lemma \ref{lem2} and by the AM-GM inequality, we have
    \begin{eqnarray*}
        |\langle Ax,x\rangle|&\leq& \langle f^2(|A|)x,x\rangle^{1/2} \langle g^2(|A^*|)x,x\rangle^{1/2}
        \leq \left \langle  \frac{f^2(|A|)+g^2(|A^*|)}{2}x,x \right\rangle.
    \end{eqnarray*}
    Therefore, using Lemma \ref{lem1}, we obtain
     \begin{eqnarray*}
        |\langle Ax,x\rangle|^2 &\leq&  \left \langle  \frac{f^2(|A|)+g^2(|A^*|)}{2}x,x \right\rangle |\langle x,A^*x\rangle| \\
        & \leq&  \frac{ \left|\left \langle A \frac{f^2(|A|)+g^2(|A^*|)}{2}x,x \right\rangle \right |+ \|A^*x\| \left\|  \frac{f^2(|A|)+g^2(|A^*|)}{2}x\right\| }{2}\\
         & \leq&  \frac{ \left|\left \langle A \frac{f^2(|A|)+g^2(|A^*|)}{2}x,x \right\rangle \right|+\frac{ \|A^*x\|^2+ \left\|  \frac{f^2(|A|)+g^2(|A^*|)}{2}x\right\|^2}{2} }{2}\\
         && \quad (\text{ by the AM-GM inequality})\\
         &\leq& \frac12 w\left(A \frac{f^2(|A|)+g^2(|A^*|)}{2} \right)+ \frac14 \left\| |A^*|^2+ \left(\frac{f^2(|A|)+g^2(|A^*|)}{2} \right)^2 \right\|.
    \end{eqnarray*}
    Taking the supremum over $\|x\|=1,$ we obtain
     \begin{eqnarray*}
         w^2(A) \leq \frac12 w\left(A \frac{f^2(|A|)+g^2(|A^*|)}{2} \right)+ \frac14 \left\| |A^*|^2+ \left(\frac{f^2(|A|)+g^2(|A^*|)}{2} \right)^2 \right\|.
     \end{eqnarray*}
     Replacing $A$ by $A^*$, we get
      \begin{eqnarray*}
         w^2(A) \leq \frac12 w\left( \frac{f^2(|A^*|)+g^2(|A|)}{2}A \right)+ \frac14 \left\| |A|^2+ \left(\frac{f^2(|A^*|)+g^2(|A|)}{2} \right)^2 \right\|,
     \end{eqnarray*}
     as desired.
\end{proof}

Setting $f(t)=t^{\alpha}$ and $g(t)=t^{1-\alpha}$,  $\alpha \in [0,1]$ in Theorem \ref{th2}, we get the following corollary.

\begin{cor}\label{cor3}
 Let $A\in \mathcal{B}(\mathcal{H})$ and $\alpha \in [0,1]$. Then
     \begin{eqnarray*}
         w^2(A) \leq \frac12 w\left(A \frac{|A|^{2\alpha}+|A^*|^{2(1-\alpha)}}{2} \right)+ \frac14 \left\| |A^*|^2+ \left(\frac{|A|^{2\alpha}+|A^*|^{2(1-\alpha)}}{2} \right)^2 \right\|
     \end{eqnarray*}
     and 
       \begin{eqnarray*}
         w^2(A) \leq \frac12 w\left( \frac{|A^*|^{2\alpha}+|A|^{2(1-\alpha)}}{2}A \right)+ \frac14 \left\| |A|^2+ \left(\frac{|A^*|^{2\alpha}+|A|^{2(1-\alpha)}}{2} \right)^2 \right\|.
     \end{eqnarray*}
\end{cor}

  In particular, setting $\alpha=\frac12$ in Corollary \ref{cor3}, we derive the following bounds.

  \begin{cor}\label{cor4}
 Let $A\in \mathcal{B}(\mathcal{H})$. Then
     \begin{eqnarray*}
         w^2(A) &\leq & \frac12 w\left(A \frac{|A|^{}+|A^*|^{}}{2} \right)+ \frac14 \left\| |A^*|^2+ \left(\frac{|A|^{}+|A^*|^{}}{2} \right)^2 \right\|
     \end{eqnarray*}
     and 
       \begin{eqnarray*}
         w^2(A) & \leq & \frac12 w\left( \frac{|A|^{}+|A^*|^{}}{2}A \right)+ \frac14 \left\| |A|^2+ \left(\frac{|A|^{}+|A^*|^{}}{2} \right)^2 \right\|.
     \end{eqnarray*}
\end{cor}

\begin{remark}
(i) The inequalities in Corollary \ref{cor4} look like \eqref{A1} and \eqref{B1}, however not the same. To support our bounds, we consider an example.
 Take a matrix $A=\begin{pmatrix}
     0&1&0&0\\
     0&0&0&1\\
     0&0&0&0\\
     0&0&1&0
 \end{pmatrix}.$ We see that 
 \begin{eqnarray*}
         \frac12 w\left(A \frac{|A|^{}+|A^*|^{}}{2} \right)+ \frac14 \left\| |A^*|^2+ \left(\frac{|A|^{}+|A^*|^{}}{2} \right)^2 \right\| &=& \frac{\sqrt{5}+\sqrt{13}}{16}+\frac12\approx 0.86510120331,
     \end{eqnarray*}
whereas 
 $$\frac12 w(|A||A^*|)+\frac14 \left\| |A|^2+|A^*|^2\right\|=1.$$
This shows that there are operators for which Corollary \ref{cor4} gives better bound than \eqref{B1}.

 (i) Clearly the bounds in Corollary \ref{cor4} always give stronger estimates than the classical upper estimate in \eqref{Eq_Intro_Equiv}.
\end{remark}

  Using the definition of numerical radius and Cauchy-Schwarz inequality, we check that 
  $$w\left(A \frac{|A|^{}+|A^*|^{}}{2} \right)\leq \frac12 \left\| |A^*|^2+ \left(\frac{|A|^{}+|A^*|^{}}{2} \right)^2 \right\|$$
    and 
    $$w\left(\frac{|A|^{}+|A^*|^{}}{2}A \right)\leq \frac12 \left\| |A|^2+ \left(\frac{|A|^{}+|A^*|^{}}{2} \right)^2 \right\|.$$
    
    Using these inequalities in Corollary \ref{cor4}, we deduce the following bound, which improves the classical upper bound in \eqref{Eq_Intro_Equiv}. 
    
    \begin{cor}\label{pintu}
        If $A\in \mathcal{B}(\mathcal{H})$, then
\begin{eqnarray*}
         w^2(A) \leq \frac12 \min \left\{  \left\| |A|^2+ \left(\frac{|A|^{}+|A^*|^{}}{2} \right)^2 \right\|, \left\| |A^*|^2+ \left(\frac{|A|^{}+|A^*|^{}}{2} \right)^2 \right\| \right\}.
     \end{eqnarray*}
\end{cor}


Now, from the proof of Theorem \ref{th2}, we can also obtain following numerical radius inequalities, from which we deduce improved bounds. 

\begin{theorem}\label{th3}
      Let $A\in \mathcal{B}(\mathcal{H})$. If $f,g$ are two continuous nonnegative functions on $[0,\infty)$ such that $f(t)g(t)=t$, for all $t\geq 0,$ then
     \begin{eqnarray*}
         w^2(A) \leq \frac12 w\left(A \frac{f^2(|A|)+g^2(|A^*|)}{2} \right)+ \frac12 \|A\| \left\| \frac{f^2(|A|)+g^2(|A^*|)}{2}  \right\|
     \end{eqnarray*}
     and 
      \begin{eqnarray*}
         w^2(A) \leq \frac12 w\left( \frac{f^2(|A^*|)+g^2(|A|)}{2}A \right)+ \frac12 \|A\| \left\|\frac{f^2(|A^*|)+g^2(|A|)}{2} \right\|.
     \end{eqnarray*}
\end{theorem}

Setting $f(t)=t^{\alpha}$ and $g(t)=t^{1-\alpha}$, $\alpha \in [0,1]$ in Theorem \ref{th3}, we get the following corollary.

\begin{cor}\label{cor5}
 Let $A\in \mathcal{B}(\mathcal{H})$ and $\alpha \in [0,1]$. Then
     \begin{eqnarray*}
         w^2(A) \leq \frac12 w\left(A \frac{|A|^{2\alpha}+|A^*|^{2(1-\alpha)}}{2} \right)+ \frac12 \|A\| \left\|  \frac{|A|^{2\alpha}+|A^*|^{2(1-\alpha)}}{2}  \right\|
     \end{eqnarray*}
     and 
       \begin{eqnarray*}
         w^2(A) \leq \frac12 w\left( \frac{|A^*|^{2\alpha}+|A|^{2(1-\alpha)}}{2}A \right)+ \frac12 \|A\| \left\|  \frac{|A^*|^{2\alpha}+|A|^{2(1-\alpha)}}{2}  \right\|.
     \end{eqnarray*}
\end{cor}

In particular,
taking $\alpha=\frac12$ in Corollary \ref{cor5}, we deduce the following refined bound.

  \begin{cor}\label{cor6}
 Let $A\in \mathcal{B}(\mathcal{H})$. Then
     \begin{eqnarray*}
         w^2(A) \leq \frac12 \min \left\{ w\left(A \frac{|A|^{}+|A^*|^{}}{2} \right),  w\left( \frac{|A|^{}+|A^*|^{}}{2}A \right) \right\}+ \frac12 \|A\| \left\|  \frac{|A|^{}+|A^*|^{}}{2} \right\|.
     \end{eqnarray*}
\end{cor}

We observe that Corollary \ref{cor6} and Corollary \ref{cor4} give different numerical radius bounds, and these bounds are not comparable, in general. To see this, consider the matrices $\begin{pmatrix}
    0&1&0\\
    0&0&0\\
    0&0&0
\end{pmatrix}$ and $\begin{pmatrix}
    0&2&0\\
    0&0&3\\
    0&0&0
\end{pmatrix}$.

\begin{remark} (i) Clearly, we have 
\begin{eqnarray*}
         w^2(A) &\leq & \frac12 \min \left\{ w\left(A \frac{|A|^{}+|A^*|^{}}{2} \right),  w\left( \frac{|A|^{}+|A^*|^{}}{2}A \right) \right\}+ \frac12 \|A\| \left\|  \frac{|A|^{}+|A^*|^{}}{2} \right\|\\
         &\leq& \frac12 \min \left\{ \left\| A \frac{|A|^{}+|A^*|^{}}{2} \right\|,  \left\| \frac{|A|^{}+|A^*|^{}}{2}A \right\| \right\}+ \frac12 \|A\| \left\|  \frac{|A|^{}+|A^*|^{}}{2} \right\|\\
         &\leq& \|A\| \left\|  \frac{|A|^{}+|A^*|^{}}{2} \right\| 
         \leq \|A\|^2.
     \end{eqnarray*}
     
Thus, Corollary \ref{cor6} refines the bound $w^2(A)\leq \|A\| \left\|  \frac{|A|^{}+|A^*|^{}}{2} \right\|,$ which follows by combining the existing bounds $w(A)\leq \|A\|$ and $w(A)\leq \left\|  \frac{|A|^{}+|A^*|^{}}{2} \right\|,$ given in \cite{Kittaneh1} (see also \cite{BHU}).

(ii) We consider a simple example to show that there are operators for which Corollary \ref{cor6} gives a stronger bound than the well known bound \eqref{EQNDRA},
given by Dragomir \cite{DRAG}. Take $A=\begin{pmatrix}
    0&1\\
    0&0
\end{pmatrix}$. We see that 
\begin{eqnarray*}
    && \frac38=\frac12 \min \left\{ w\left(A \frac{|A|^{}+|A^*|^{}}{2} \right),  w\left( \frac{|A|^{}+|A^*|^{}}{2}A \right) \right\}+ \frac12 \|A\| \left\|  \frac{|A|^{}+|A^*|^{}}{2} \right\|\\
    &<& \frac12=\frac12w(A^2)+\frac12 \|A\|^2.
\end{eqnarray*}


(iii) From the inequalities in (i), we deduce that if $w(A)=\|A\|$ (e.g. if $A$ is normal), then $$ w\left(A \frac{|A|^{}+|A^*|^{}}{2} \right)=  w\left( \frac{|A|^{}+|A^*|^{}}{2}A \right) = w(A|A|)=w(A|A^*|)=w(|A|A)=w(|A^*|A)=\|A\|^2.$$
\end{remark}

\subsection{Bounds via Euclidean operator radius}
Suppose $A,B\in\mathcal{B}(\mathcal{H})$.
Recall that the Euclidean operator norm and Euclidean operator radius of a $2$-tuple $(A,B)$ are defined, respectively, as (see \cite{Popescu_MAMS_2009}) 
$$\|A,B\|_e=\sup \left\{ \sqrt{|\left<Ax,y\right>|^2+|\left<Bx,y\right>|^2}: {\|x\|=\|y\|=1} \right\}$$
and 
$$w_e(A,B)=\sup \left\{ \sqrt{|\left<Ax,x\right>|^2+|\left<Bx,x\right>|^2} : {\|x\|=1} \right\}.$$

One more generalization, called the $p$-numerical radius of a $2$-tuple $(A,B)$, is defined as  (see \cite{MOS}) $$w_p(A,B)=\sup \left\{ \left(|\left<Ax,x\right>|^p+|\left<Bx,x\right>|^p\right)^{1/p} : {\|x\|=1} \right\},$$
where $p\geq 1.$

In this subsection, we obtain numerical radius inequalities for the sums and products of operators involving the Euclidean operator norm and Euclidean operator radius.
Using the notion of Euclidean operator norm, first we obtain the following theorem.

\begin{theorem}\label{th1}
    If $T\in \mathcal{B}(\mathcal{H}),$ then
    $$w^2(A)\leq \frac14 w(A^2)+ \frac14 \|A, A^*\|_e^2+ \frac18 \|A^*A+AA^*\|.$$
\end{theorem}
\begin{proof}
    Take $x,y\in \mathcal{H}$ with $\|x\|=\|y\|=1$. We have
    $$\langle Re(A)x,y\rangle= \frac{\langle Ax,y\rangle+\langle A^*x,y\rangle.}{2}$$
    Therefore, 
    \begin{eqnarray*}
         |\langle Re(A)x,y\rangle|^2 &\leq&   \frac{1}{4} \left(|\langle Ax,y\rangle|^2+|\langle A^*x,y\rangle|^2+ 2|\langle Ax,y\rangle \langle A^*x,y\rangle|  \right)\\
         &\leq&   \frac{1}{4} \left(|\langle Ax,y\rangle|^2+|\langle A^*x,y\rangle|^2+ \|Ax\| \|A^*x\|+  |\langle Ax,A^*x\rangle|  \right) \\
         && \quad (\textit{by Lemma \ref{lem1}})\\
         &\leq&  \frac{1}{4} \left(|\langle Ax,y\rangle|^2+|\langle A^*x,y\rangle|^2+ \frac12 (\|Ax\|^2+ \|A^*x\|^2)+  |\langle A^2x,x\rangle|  \right)\\
         && \quad (\textit{by AM-GM inequality})\\
         &\leq& \frac14 \left( \|A,A^*\|_e^2+ \frac12\|A^*A+AA^*\|+w(A^2) \right).
    \end{eqnarray*}
    Taking the supremum over $\|x\|=\|y\|=1,$ we get
    $$\|Re(A)\|^2 \leq \frac14 \left( \|A,A^*\|_e^2+ \frac12\|A^*A+AA^*\|+w(A^2) \right).$$
    Replacing $A$ by $e^{i\theta} A$, $\theta \in \mathbb{R}$, we get
     $$\|Re(e^{i\theta}A)\|^2 \leq \frac14 \left( \|A,A^*\|_e^2+ \frac12\|A^*A+AA^*\|+w(A^2) \right).$$
     Taking the supremum over $\theta \in \mathbb{R},$ we get
     $$w^2(A)= \sup_{\theta\in \mathbb{R}}\|Re(e^{i\theta}A)\|^2 \leq \frac14 \left( \|A,A^*\|_e^2+ \frac12\|A^*A+AA^*\|+w(A^2) \right),$$ as desired.
\end{proof}

Similarly to Theorem \ref{th1}, we can also prove the following numerical radius bound.

\begin{cor}\label{cor1}
      If $T\in \mathcal{B}(\mathcal{H}),$ then
    $$w^2(A)\leq \frac14 w(A^2)+ \frac14 w_e^2(A, A^*)+ \frac18 \|A^*A+AA^*\|.$$
\end{cor}

 \begin{remark}
   (i)  It is easy to check $w_e(A, A^*)\leq \|A,A^*\|_e$ and from the Cauchy-Schwarz inequality it follows that $\|A,A^*\|_e^2 \leq  \|A^*A+AA^*\|.$ Also, from the power inequality of the numerical radius and \eqref{Kittaneh_Eq_1}, we have $w(A^2)\leq w^2(A) \leq \frac12  \|A^*A+AA^*\| .$ Thus,
     \begin{eqnarray*}
      w^2(A) &\leq&  \frac14 w(A^2)+ \frac14 w_e^2(A, A^*)+ \frac18 \|A^*A+AA^*\| \\
      &\leq&   \frac14 w(A^2)+ \frac14 \|A, A^*\|_e^2+ \frac18 \|A^*A+AA^*\| \\
       &\leq& \frac12 \|A^*A+AA^*\| \\
       &\leq&  \|A\|^2. 
     \end{eqnarray*}
     Therefore, both the numerical radius bounds in Theorem \ref{th1} and Corollary \ref{cor1} refine the well known bound $w^2(A) \leq \frac12  \|A^*A+AA^*\|$, proved by Kittaneh \cite{Kittaneh2}. To see proper refinement, one can take a simple matrix $A=\begin{pmatrix}
         0&1\\
         0&0
     \end{pmatrix}.$

     (ii) From the above inequalities, we conclude if $w(A)=\|A\|$ (e.g. if $A$ is normal), then $$\|A\|= \frac1{\sqrt{2}} \|A, A^*\|_e
    .$$
 \end{remark}


Next, we need the following lemmas to present more results in this direction. The first lemma is known as McCarthy inequality.

 \begin{lemma} \cite{Kittaneh88} \label{lem3-}
        Let $A\in \mathcal{B}(\mathcal{H})$ be positive. Then 
        $$\langle Ax,x \rangle^p\leq \langle A^px,x \rangle,$$
        for all $p\geq 1$ and for all $x\in \mathcal{H}$ with $\|x\|=1.$
    \end{lemma}

The second lemma is also an inner product inequality  involving a $2\times 2$ positive operator matrix.

    \begin{lemma}\cite{Kittaneh88} \label{lem5}
    Let $A,B,C\in\mathcal{B(\mathcal{H}})$ where $A$ and $B$ are positive.Then  the operator matrix 
$\begin{bmatrix}
    A & C^*\\
    C &  B
\end{bmatrix}$
    $\in\mathcal{B(\mathcal{H\oplus H)}}$ is positive if and only if \begin{eqnarray*}
        |\langle Cx,y\rangle|^2\leq \langle Ax,x\rangle\langle By,y\rangle.
    \end{eqnarray*}
    \end{lemma}

The third lemma is known as Bohr’s inequality, which is for $n$ positive real numbers.

    \begin{lemma} \cite{V} \label{lem4}
    Suppose $a_1, a_2, \ldots, a_n$ are positive real numbers. Then
    \begin{eqnarray*}
        \left(\sum_{i=1}^{n}a_i\right)^r\leq n^{r-1}\sum_{i=1}^{n}a_i^r,
    \end{eqnarray*}
    for all $r\geq1$.
    \end{lemma}

Here we also recall here the notion of double convex function (see \cite{Bourin}).
A function $h:[0,\infty)\to [0,\infty)$ is said to be double convex if $h(t)$ is convex and $h(t)$ is geometrically convex, i.e., $h(\sqrt{ab})\leq \sqrt{h(a)h(b)}$ for all $a,b\in [0,\infty).$

Now, using the notion of double convex function and above lemmas, we present the numerical radius inequalities.

    \begin{theorem}\label{theorem4}
        Let $A,B,C\in\mathcal{B(\mathcal{H}})$, where $A$ and $B$ are positive, and let $h$ be an increasing double convex function on $[0,\infty)$. If the operator matrix
$\begin{bmatrix}
    A & C^*\\
    C &  B
\end{bmatrix}$ 
is postive, then
\[h(w(C))\leq \sqrt{\frac{\alpha\beta}{2}} w_e\left(\frac{h(A)}{\alpha},\frac{h(B)}{\beta}\right),\]
for all $\alpha, \beta>0$.
\end{theorem}

\begin{proof}
To prove this result first we note the following inequality.
For any two positive real numbers $a$ and $b$, we have 
\begin{eqnarray}\label{EQN-g}
    \sqrt{ab}\le \frac{\sqrt{\alpha\beta}}{2}\left(\frac{a}{\alpha}+\frac{b}{\beta}\right),
\end{eqnarray}
where $\alpha,\beta>0$. This follows by replacing $a$ and $b$ by $\frac{a}{\alpha}$ and $\frac{b}{\beta}$, respectively, in 
 $\sqrt{ab} \leq \frac{a+b}{2}$.
Now, using Lemma \ref{lem5} and increasing double convex function $h$, we obtain
 \begin{eqnarray*}
       h( |\langle Cx,x\rangle|)&\leq &h(\sqrt{\langle Ax,x\rangle\langle Bx,x\rangle})\\
       &\leq &\sqrt{h(\langle Ax ,x\rangle\langle Bx,x\rangle})\\
       &\leq &\sqrt{\langle h(A)x ,x\rangle\langle h(B)x,x\rangle}\\
       &\leq & \sqrt{\frac{\alpha\beta}{2}}\left(\sqrt{\frac{|\langle h(A)x,x\rangle|^2}{\alpha^2}+\frac{|\langle h(B)x,x\rangle|^2}{\beta^2}}\right) \quad (\text{by \eqref{EQN-g}})\\
        &=&\sqrt{\frac{\alpha\beta}{2}}\left(\sqrt{\left| \left\langle\frac{h(A)x}{\alpha},x\right\rangle \right|^2+\left|\left\langle \frac{h(B)x}{\beta},x\right\rangle\right|^2}\right).
        \end{eqnarray*}
        Taking the supremum over $\|x\|=1$, we get the desired inequality.        
    \end{proof}

For every $B,C \in \mathcal{B}(\mathcal{H})$, we observe that 
$
\begin{bmatrix}
BB^* & BC \\
C^*B^* & C^*C
\end{bmatrix}
\geq 0$.
Using this operator matrix in Theorem \ref{theorem4}, we obtain the following corollary.

  \begin{cor}
Let $B,C \in \mathcal{B}(\mathcal{H})$, and let $h$ be a double convex 
increasing function on $[0,\infty)$.  Then
\[
h(w(BC))
\le
\sqrt{\frac{\alpha\beta}{2}}\;
w_e\!\left(
\frac{h(BB^*)}{\alpha},
\frac{h(C^*C)}{\beta}
\right),
\]
for all $\alpha,\beta>0$.
In particular, for $h(t)=t$, 
\begin{eqnarray}\label{EQN3-}
    w(BC)
\le
\min_{\alpha, \beta>0}\sqrt{\frac{\alpha\beta}{2}}\;
w_e\!\left(
\frac{BB^*}{\alpha},
\frac{C^*C}{\beta}
\right).
\end{eqnarray}
  \end{cor}

Suppose $A=U|A|$ is the polar decomposition of $A$. Setting $B=U|A|^{1-t}$ and $C=|A|^t$ (where $t\in [0,1]$) in \eqref{EQN3-}, we obtain the following bound.

\begin{cor}\label{PN}
Let $A\in \mathcal{B}(\mathcal{H})$. Then
    \begin{eqnarray*} 
    w(A)
\le
\min_{\alpha, \beta>0}\sqrt{\frac{\alpha\beta}{2}}\;
w_e\!\left(
\frac{1}{\alpha}|A^*|^{2(1-t)},
\frac{1}{\beta}|A|^{2t}
\right),
\end{eqnarray*}
for all $t\in [0,1]$.
\end{cor}

We remark that Corollary \ref{PN} refines the existing bound $w^2(A) \leq \frac12 \left\| |A|^2+|A^*|^2\right\|,$ see Remark \ref{REM1}, below.

Next inequality also for the products of operators, reads as follows.

\begin{theorem}\label{th1}
    Let $A,B\in \mathcal{B}(\mathcal{H})$ be such that $|A|B=B^*|A|$. Let $f$, $g$ be non negative continuous functions on $[0,\infty)$ such that $f(\lambda)g(\lambda)=\lambda$ for all $\lambda\in[0,\infty)$.
    Then 
  \begin{eqnarray}\label{EQN2}
    w(AB)  \leq \frac{r(B)}{\sqrt{2}}w_e\left(f^2(|A|),g^2(|A^*|)\right).
\end{eqnarray}
In particular,
\begin{eqnarray}\label{EQN3}
    w(AB)\leq \frac{r(B)}{\sqrt{2}}w_e\left(|A|^{2t},|A^*|^{2(1-t)}\right), \quad \text{for all $t\in[0,1].$}
\end{eqnarray}
\end{theorem}

\begin{proof}
Take $x \in  \mathcal{H}$ with $\|x\|=1.$
Using Lemma \ref{lem2}, we have
    \begin{eqnarray*}
        |\langle ABx,x\rangle| &\leq&  r(B) \left(\langle f(|A|)x,f(|A|)x\rangle \langle f(|A^*|)x,f(|A^*|)x\rangle \right)^\frac{1}{2} \\ &\leq & \frac{r(B)}{\sqrt{2}}\left(|\langle f^2(|A|)x,x\rangle|^2+|\langle g^2(|A^*|)x,x\rangle|^2\right)^\frac{1}{2}.
\end{eqnarray*}
Taking the supremum over $\|x\|=1$, we get the desired first inequality \eqref{EQN2}.
In particular, setting $f(\lambda)=\lambda^{t}$ and $g(\lambda)=\lambda^{1-t}$ in \eqref{EQN2}, we obtain the desired second inequality \eqref{EQN3}.
 \end{proof}

Setting $B=I$ in Theorem \ref{th1}, we obtain the following novel upper bound for the numerical radius.

\begin{cor}\label{cor1}
     Let $A\in \mathcal{B}(\mathcal{H})$. Then 
    \begin{eqnarray}\label{eqn1}
            w(A) \leq \frac{1}{\sqrt{2}}w_e\left(|A|^{2t},|A^*|^{2(1-t)}\right),
    \end{eqnarray}
    for all $t\in [0,1].$ 
\end{cor}

\begin{remark}\label{REM1}
    Using Lemma \ref{lem3-}, we see that 
     \begin{eqnarray*}
            w(A) \leq \frac{1}{\sqrt{2}}w_e\left(|A|^{2t},|A^*|^{2(1-t)}\right) \leq \frac{1}{\sqrt{2}} \left\| |A|^{4t}+ |A^*|^{4(1-t)} \right\|^{1/2},
    \end{eqnarray*}
    for all $t\in [0,1].$ In particular, for $t=\frac12,$
 \begin{eqnarray*}
            w^2(A) \leq \frac{1}{{2}}w_e^2\left(|A|^{},|A^*|^{}\right) \leq \frac{1}{{2}} \left\| |A|^{2}+ |A^*|^{2} \right\|^{}.
    \end{eqnarray*}
    Therefore, Theorem \ref{theorem4} and Theorem \ref{th1} refine and generalize the well known second inequality in \eqref{Kittaneh_Eq_1}.
\end{remark}

Next we develop numerical radius inequalities for the sums of product operators involving $p$-numerical radius of $2$-tuples of operators.

 \begin{theorem}\label{THM1}
     Let $A_i, B_i, X_i \in \mathcal{B(\mathcal{H)}}$ for $i=1,2, \ldots,n$. Let $f$, $g$ be two non negative continuous functions on $[0,\infty)$ such that $f(t)g(t)=t$, for all $t\in[0,\infty)$. Then
     \begin{eqnarray*}
w^p\left(\sum_{i=1}^{n}A_i^*X_iB_i\right)\leq\frac{n^{p-1}}{2}\left(\sum_{i=1}^{n}w_p^p(B_i^*f^2(|X_i|)B_i,A_i^*g^2(|X_i^*|)A_i)\right),
\end{eqnarray*}
where $p\geq 1.$
 \end{theorem}
 
 \begin{proof} Take $x\in \mathcal{H}$ with $\|x\|=1.$ We have
 \begin{eqnarray*}
\left|\left\langle\left(\sum_{i=1}^{n}A_i^*X_iB_i\right)x,x \right\rangle \right|^p&\leq&\left(\sum_{i=1}^{n}|\left\langle (A_i^*X_iB_i)x,x\right\rangle|\right)^p \\
     &\leq& n^{p-1}\left(\sum_{i=1}^{n}|\langle(A_i^*X_iB_i)x,x\rangle|^p\right)\\
     &=&n^{p-1}\left(\sum_{i=1}^{n}|\langle(X_iB_i)x,A_ix\rangle|^p\right)\\
     &\leq& n^{r-1}\left(\sum_{i=1}^{n}\|f(|X_i|)\|B_ix\|^r\|g(|X_i^*|)A_ix\|^p\right) \quad (\textit{by Lemma \ref{lem2}})\\
    &=&n^{p-1}\left(\sum_{i=1}^{n}\langle f^2(|X_i|)B_ix,B_ix\rangle^\frac{p}{2}\langle g^2(|X_i^*|)A_ix,A_ix\rangle^\frac{p}{2} \right)\\
     &\leq&\frac{n^{p-1}}{2}\left(\sum_{i=1}^{n}\langle B_i^*f^2(|X_i|)B_ix,x\rangle^p+\langle A_i^*g^2(|X_i^*|)A_ix,x\rangle^p \right).\\
\end{eqnarray*}
      Taking the supremum over $\|x\|=1$, we get the desired inequality.      
      \end{proof}
      
    \begin{remark}\label{remark2}
    We see that
    \begin{eqnarray*} &&w^r\left(\sum_{i=1}^{n}A_i^*X_iB_i\right)\\&\leq&\frac{n^{p-1}}{2}\left(\sum_{i=1}^{n}w_p^p(B_i^*f^2(|X_i|)B_i,A_i^*g^2(|X_i^*|)A_i)\right)\\
    &=&\frac{n^{p-1}}{2}\sup_{\|x\|=1}\sum_{i=1}^{n} \left(\langle B_i^*f^2(|X_i|)B_ix,x\rangle^p+\langle A_i^*g^2(|X_i^*|)A_ix,x\rangle^p \right)\\ 
    &\leq&
    \frac{n^{p-1}}{2}\sup_{\|x\|=1}\sum_{i=1}^{n} \left(\langle [B_i^*f^2(|X_i|)B_i]^px,x\rangle+\langle [A_i^*g^2(|X_i^*|)A_i]^p x,x\rangle \right) (\textit{by Lemma \ref{lem3-}})\\
    &=&
    \frac{n^{p-1}}{2}\sup_{\|x\|=1}\left(\sum_{i=1}^{n}\langle [B_i^*f^2(|X_i|)B_i]^px,x\rangle+ \sum_{i=1}^{n}\langle [A_i^*g^2(|X_i^*|)A_i]^p x,x\rangle \right) \\
     &\leq&
    \frac{n^{p-1}}{\sqrt{2}}\sup_{\|x\|=1}\left( \left | \sum_{i=1}^{n}  \langle [B_i^*f^2(|X_i|)B_i]^px,x\rangle+ i \sum_{i=1}^{n} \langle [A_i^*g^2(|X_i^*|)A_i]^p x,x\rangle\right| \right) \\
      &\leq&\frac{n^{p-1}}{\sqrt2}w \left(\sum_{i=1}^{n} [B_i^*f^2(|X_i|)B_i]^p+i [A_i^*g^2(|X_i^*|)A_i]^p\right), 
     \end{eqnarray*}
     where the third inequality follows by the inequality $|a+b|\leq \sqrt{2}|a+ib|$ for $a, b\in \mathbb{R}.$
     Hence, Theorem \ref{THM1} refines \cite[Theorem 2.11]{Bhunia-BSM}, namely,
     \begin{eqnarray*} w^r\left(\sum_{i=1}^{n}A_i^*X_iB_i\right)\leq \frac{n^{p-1}}{\sqrt2}w \left(\sum_{i=1}^{n} [B_i^*f^2(|X_i|)B_i]^p+i [A_i^*g^2(|X_i^*|)A_i]^p\right) .
     \end{eqnarray*}
     \end{remark}
  
 Setting $f(\lambda)=\lambda^{t}$ and $g(\lambda)=\lambda^{1-t},\, t\in [0,1]$ in Theorem \ref{THM1}, we obtain the following corollary.

 \begin{cor}\label{cor3-}
 Let $A_i, B_i, X_i \in \mathcal{B(\mathcal{H)}}$ for $i=1,2,\ldots,n$. Then
 \begin{eqnarray*}
  w^p\left(\sum_{i=1}^{n}A_i^*X_iB_i\right)\leq\frac{n^{p-1}}{{2}}\left(\sum_{i=1}^{n}w_p^p(B_i^*|X_i|^{2t}B_i,A_i^*|X_i^*|^{2(1-t)}A_i)\right),
     \end{eqnarray*}
 for all $p\geq1$.
     \end{cor}
     
Again, setting $A_i, B_i=I$ in Theorem \ref{THM1}, we derive the following bound for the sums of operators.

 \begin{cor} \label{cor4-}
Let $X_i \in \mathcal{B(\mathcal{H)}}$ for $i=1,2,\ldots,n$. Let $f$, $g$ be two non negative continuous functions on $[0,\infty)$ so that $f(t)g(t)=t$ for all $t\in[0,\infty)$. Then
     \begin{eqnarray*}
         w^p\left(\sum_{i=1}^{n}X_i\right)\leq\frac{n^{p-1}}{2}\left(\sum_{i=1}^{n}w_p^p( f^2(|X_i|),g^2(|X_i^*|))\right).
\end{eqnarray*}
 \end{cor}
 
 Setting $n=1$  and $f(t)=g(t)=t^{\frac{1}{2}}$ in Corollary \ref{cor4-}, we get the following upper bound.
 \begin{cor}\label{cor5-}
 Let $A\in\mathcal{B(\mathcal{H})}$. Then
 \begin{eqnarray*}
         w\left(A\right)\leq\frac{1}{2^{1/p}}w_p(|A|,|A^*|),
         \end{eqnarray*}
         for all $p\geq 1.$
 \end{cor}

\begin{remark}
   Using the inequality $|a+b|\leq \sqrt{2}|a+ib |$ for any $a,b \in \mathbb{R}$, we see that 
    \begin{eqnarray*}
         w^p\left(A\right) &\leq& \frac{1}{2}w_p^p(|A|,|A^*|)\\
&=& \frac12 \sup_{\|x\|=1} \left( \langle |A|x,x \rangle^p+ \langle |A^*|x,x\rangle^p\right)\\
&\leq & \frac12 \sup_{\|x\|=1} \left( \langle |A|^px,x \rangle+ \langle |A^*|^px,x\rangle\right) \quad (\text{by Lemma \ref{lem3-}})\\
&\leq & \frac1{\sqrt{2}} \sup_{\|x\|=1} \left| \langle |A|^px,x \rangle+ i \langle |A^*|^px,x\rangle\right|\\
         &\leq& \frac{1}{\sqrt{2}}w \left( |A|^p+i|A^*|^p \right)^{}, \quad \text{ for all $p\geq 1.$ }
         \end{eqnarray*}
        Therefore, Corollary \ref{cor5-} refines and generalizes
         $w\left(A\right) \leq \frac{1}{\sqrt{2}}w \left( |A|+i|A^*| \right)^{} $, which was given in \cite[Corollary 2.15]{Bhunia-BSM}.
       \end{remark}


\section{Lower bounds for the numerical radius}

Using the Cartesian decomposition, we obtain in this section lower bounds for the numerical radius of bounded linear operators, which refine the existing lower bounds in \eqref{Eq_Intro_Equiv} and \eqref{Kittaneh_Eq_1}. To present these bounds, we first note the following inequality for nonzero vectors $x,y$ in a normed linear space $\mathcal{X}=(\mathcal{X}, \| \cdot \|),$ given by Maligranda \cite{Mali}:
 \begin{eqnarray}\label{p0}
     \|x+y\| \leq \|x\|+\|y\|-\left( 2- \left\|\frac{x}{\|x\|}+\frac{y}{\|y\|} \right\|\right)\min \left(\|x\|, \|y\| \right).
 \end{eqnarray}
 By applying Maligranda's inequality, we obtain lower bounds. The first bound is as follows, which improves $  \frac{1}{2}\|A\| \leq w(A)$.

\begin{prop}\label{prop1}
    If $A\in \mathcal{B}(\mathcal{H})$ with $Re(A)\pm Im(A) \neq 0$, then
    \begin{eqnarray*}\label{}
        \frac{1}{2}\|A\|+ \mu(A) &\leq & w(A),
    \end{eqnarray*}
    where $\mu(A)=\frac1{2\sqrt{2}} \left(2-\left\|\frac{Re(A)+Im(A)}{\| Re(A)+Im(A)\|}+ i \frac{Re(A)-Im(A)}{\| Re(A)-Im(A)\|} \right\| \right) \min \left( \| Re(A) \pm Im(A)\| \right).$ 
    \end{prop}

\begin{proof}
Following \cite[Theorem 2.1]{Bhunia-RMJM}, we obtain that $$w(A) \geq \frac1{\sqrt{2}} \max \left\{ \| Re(A) + Im(A)\|, \| Re(A) - Im(A)\| \right\}.$$
Setting the normed linear space $\mathcal{X}=(\mathcal{B}(\mathcal{H}), \|\cdot\|)$ and the vectors $x=\frac{1}{\sqrt{2}}(Re(A)+Im(A))$ and $y=\frac{i}{\sqrt{2}} (Re(A)-Im(A))$ in \eqref{p0}, and using the above inequality, we obtain the desired bound.
\end{proof}

\begin{remark}
    Since $\mu(A)\geq 0$ for every $A\in \mathcal{B}(\mathcal{H})$, Proposition \ref{prop1} refines the classical numerical radius bound $\frac{1}{2}\|A\|\leq w(A).$
To show the proper refinement, we consider a simple example. Suppose $A=\begin{pmatrix}
    1&0\\
    0&i
\end{pmatrix}.$ We see that  \begin{eqnarray*}\label{}
    \frac{1}{2}\|A\|=\frac12< \frac{1}{\sqrt{2}} =    \frac{1}{2}\|A\|+ \mu(A) &< & 1=w(A).
    \end{eqnarray*}
\end{remark}

Following \cite[Remark 2.2]{Bhunia-RMJM}, we have if $w(A)=\frac{1}{2}\|A\|,$ then $\|Re(A)+ Im(A)\|=\|Re(A)- Im(A)\|.$ Using this and Proposition \ref{prop1}, we deduce a necessary condition for the equality $w(A)=\frac{1}{2}\|A\|.$

\begin{cor}
    Suppose $A\in \mathcal{B}(\mathcal{H})$ with $\|Re(A)\pm Im(A)\|\neq 0$. If $w(A)=\frac{1}{2}\|A\|,$ then $$\|Re(A)\pm Im(A)\|=\frac{1}{\sqrt{2}}\|A\|.$$ 
\end{cor}


Similarly to Proposition \ref{prop1}, we also obtain the following numerical radius bound, which refines the bound $ \frac{1}{4}\|A^*A+AA^*\| 
         \leq w^2(A)$.

\begin{prop}\label{prop2}
    If $A\in \mathcal{B}(\mathcal{H})$ with $Re(A)\pm Im(A) \neq 0$, then
    \begin{eqnarray*}\label{}
        \frac{1}{4}\|A^*A+AA^*\| 
        + \nu(A) &\leq& w^2(A),
    \end{eqnarray*}
    where $\nu(A)=\frac1{4} \left(2-\left\|\frac{(Re(A)+Im(A))^2}{\| Re(A)+Im(A)\|^2}+ i \frac{(Re(A)-Im(A))^2}{\| Re(A)-Im(A)\|^2} \right\| \right) \min \left( \| Re(A) \pm Im(A)\|^2 \right).$
    \end{prop}

\begin{remark}
    Since $\nu(A)\geq 0$ for every $A\in \mathcal{B}(\mathcal{H})$, Proposition \ref{prop2} refines the numerical radius bound $\frac{1}{4}\|A^*A+AA^*\| \leq w^2(A),$ proved by Kittaneh \cite{Kittaneh2}.
    We consider a simple example to show proper refinement. Suppose $A=\begin{pmatrix}
    1&0\\
    0&i
\end{pmatrix}.$ We see that  \begin{eqnarray*}\label{}
    \frac{1}{4}\|A^*A+AA^*\| =\frac12< \frac12+\left(\frac12-\frac{1}{2\sqrt{2}}\right) =     \frac{1}{4}\|A^*A+AA^*\| 
        + \nu(A) <1= w^2(A).
    \end{eqnarray*}
    Recently, Jana et al. \cite{Jana} proved that 
    \begin{eqnarray}\label{pin}
        \frac{1}{4}\|A^*A+AA^*\|+\frac{\mu}{2}\max \{\| Re(A)\|, \| Im(A)\| \} \leq w^2(A),
    \end{eqnarray}
    where $\mu= \left|\| Re(A)+Im(A) \|- \| Re(A)-Im(A)\| \right|.$ 
    Considering the above example, we conclude that there are operators for which Proposition \ref{prop2} gives a stronger estimate than estimate \eqref{pin}.
    \end{remark}


Similarly to Proposition \ref{prop1}, we also obtain the following other lower bounds, which are also recently studied in  \cite{Bhunia-CMB}.  

\begin{cor}
    If $A\in \mathcal{B}(\mathcal{H})$ with $Re(A)\neq 0 $ and $ Im(A) \neq 0$, then
\begin{eqnarray}\label{pintu-001}
        \frac{1}{2}\|A\|+ \gamma(A) &\leq & w(A),
    \end{eqnarray}
    where $\gamma(A)=\frac12 \left(2-\left\|\frac{Re(A)}{\| Re(A)\|}+ i \frac{Im(A)}{\| Im(A)\|} \right\| \right) \min \left( \| Re(A)\|, \| Im(A)\| \right)$
    and 
\begin{eqnarray}\label{pintu-002}
       \frac{1}{4}\|A^*A+AA^*\| 
       + \delta(A) & \leq & w^2(A),
    \end{eqnarray}
    where $\delta(A)=\frac12 \left(2-\left\|\frac{Re^2(A)}{\| Re(A)\|^2}+ i \frac{Im^2(A)}{\| Im(A)\|^2} \right\| \right) \min \left( \| Re(A)\|^2, \| Im(A)\|^2 \right).$
\end{cor}


Next, applying Proposition \ref{prop2}, we obtain an upper bound for the numerical radius of the generalized commutators and anti-commutators of operators, which refines (and generalizes) the inequality $w(AB\pm BA) \leq 2\sqrt{2} w(A) \|B\|,$ given by Fong and Holbrook \cite{Fong}.  


\begin{cor}\label{cor7}
    Let $A,B,X,Y\in \mathcal{B}(\mathcal{H}).$ Then
    \begin{eqnarray*}
        w(AXB\pm BYA) &\leq& 2\sqrt{2} \max\{\|XB\|, \|BY\| \} \sqrt{w^2(A)-\nu(A)},
    \end{eqnarray*}
    
    where $\nu(A)$ is the same as Proposition \ref{prop2}. 
    
    In particular, for $X=Y=I$,
     \begin{eqnarray}\label{pi-1}
        w(AB\pm BA) &\leq& 2\sqrt{2} \|B\| \sqrt{w^2(A)-\nu(A)}.
    \end{eqnarray}

    Interchanging $A$ and $B$, we get
     \begin{eqnarray}
        w(AB\pm BA) &\leq& 2\sqrt{2} \|A\| \sqrt{w^2(B)-\nu(B)}.
    \end{eqnarray}
\end{cor}

\begin{proof}
Take $x\in {\mathcal{H}}$ with $\|x\|=1$. Assuming $\|X\|\leq 1$ and $\|Y\|\leq 1$, we obtain 
	\begin{eqnarray*}
		|\langle (AX\pm YA)x,x\rangle|
		&\leq& \|A^*x\|+ \|Ax\| 
		\leq   \sqrt{2(\|A^*x\|^2+ \|Ax\|^2)}\\ 
		&\leq&  \sqrt{2\|AA^*+A^*A\|}
		\leq 2\sqrt{2}\sqrt{ w^2(A)-\nu(A)  } \quad (\textit{by Proposition \ref{prop2}}).
	\end{eqnarray*}
Taking the supremum over $\|x\|=1$, we get
	\begin{eqnarray}\label{eqnth1}
	w(AX\pm YA)&\leq& 2\sqrt{2}\sqrt{ w^2(A)-\nu(A) }.
	\end{eqnarray}
	We now consider general $X,Y$, where  $\max  \left\{\|X\|,\|Y\| \right\}\neq 0,$ otherwise Corollary \ref{cor7} holds trivially.   
Substituting $X$ by $\frac{X}{\max  \left\{\|X\|,\|Y\| \right\}}$ and $Y$  by $\frac{Y}{\max  \left\{\|X\|,\|Y\| \right\}}$ in (\ref{eqnth1}), we obtain
	\begin{eqnarray}\label{eqnth2}
	w(AX\pm YA)\leq 2\sqrt{2}\max  \left\{\|X\|,\|Y\| \right\}\sqrt{ w^2(A)-\nu(A) }.
	\end{eqnarray}
Again, substituting $X$ by $XB$ and $Y$ by $BY$ in (\ref{eqnth2}), we get
	\begin{eqnarray*}
		w(AXB\pm BYA)\leq 2\sqrt{2}\max  \left\{\|XB\|,\|BY\| \right\}\sqrt{ w^2(A)-\nu(A) },
	\end{eqnarray*}
	 as desired.
\end{proof}

Again, applying the inequality \eqref{pintu-002}, we can obtain the following numerical radius bounds for the generalized commutators and anti-commutators of operators.

\begin{cor}\label{cor8}
    Let $A,B,X,Y\in \mathcal{B}(\mathcal{H}).$ Then
    \begin{eqnarray*}
        w(AXB\pm BYA) &\leq& 2\sqrt{2} \max\{\|XB\|, \|BY\| \} \sqrt{w^2(A)-\delta(A)},
    \end{eqnarray*}
    
    where $\delta(A)$ is the same as \eqref{pintu-002}. 
    
    In particular, for $X=Y=I$,
     \begin{eqnarray}\label{pi-2}
        w(AB\pm BA) &\leq& 2\sqrt{2} \|B\| \sqrt{w^2(A)-\delta(A)}.
    \end{eqnarray}

    Interchanging $A$ and $B$, we get
     \begin{eqnarray}
        w(AB\pm BA) &\leq& 2\sqrt{2} \|A\| \sqrt{w^2(B)-\delta(B)}.
    \end{eqnarray}
\end{cor}

\begin{remark}
   (i) Clearly, the inequalities in Corollary \ref{cor7} and Corollary \ref{cor8} refine and generalize the inequality $w(AB\pm BA) \leq 2\sqrt{2} w(A) \|B\|,$ which was established by Fong and Holbrook \cite{Fong}. To show proper refinement, we take a simple matrix $A=\begin{pmatrix}
       1+2i&0\\
       0&0
   \end{pmatrix}$. Then we see that $\nu(A)=\frac12 (1-\frac{1}{\sqrt{2}})$ and $\delta(A)=1-\frac{1}{\sqrt{2}}>0$.

   (ii) Following the inequalities \eqref{pi-1} and \eqref{pi-2}, we conclude if $w(AB\pm BA) = 2\sqrt{2} w(A) \|B\|,$ then $\left\|\frac{(Re(A)+Im(A))^2}{\| Re(A)+Im(A)\|^2}+ i \frac{(Re(A)-Im(A))^2}{\| Re(A)-Im(A)\|^2} \right\|=\left\|\frac{Re^2(A)}{\| Re(A)\|^2}+ i \frac{Im^2(A)}{\| Im(A)\|^2} \right\|=2.$
    \end{remark}

\vskip0.2 cm

\medskip
\noindent \textit{Conflict of Interest Statement.} The authors state that there is no conflict of interest. 

\medskip
\noindent\textit{Data Availability Statement.} Data sharing is not applicable to this article as no datasets were generated or analysed during the current study.

\medskip

\bibliographystyle{amsplain}

\begin{thebibliography}{99}



\bibitem{Abu-Omar-Kittaneh} A. Abu-Omar and F. Kittaneh, Upper and lower bounds for the numerical radius with an application to involution operators, Rocky Mountain J. Math. 45 (2015), no. 4, 1055–1065.

\bibitem{Feki} N. Altwaijry, C. Conde, S. S. Dragomir and K. Feki, Further norm and numerical radii inequalities for operators involving a positive operator, AIMS Math. 10 (2025), no. 2, 2684–2696.


\bibitem{Bhunia-LAA} P. Bhunia, S. Bag and K. Paul, Numerical radius inequalities and its applications in estimation of zeros of polynomials,
Linear Algebra Appl. 573 (2019), 166–177.

\bibitem{Bhunia-BSM} P. Bhunia and K. Paul,
New upper bounds for the numerical radius of Hilbert space operators,
Bull. Sci. Math. 167 (2021), Paper No. 102959, 11 pp.

\bibitem{BP} P. Bhunia and K. Paul, Development of inequalities and characterization of equality conditions for the numerical radius,
Linear Algebra Appl. 630 (2021), 306–315.

\bibitem{Bhunia-book} P.~Bhunia, S.~S.~Dragomir, M.~S.~Moslehian and K.~Paul,
Lectures on numerical radius inequalities,
Infosys Sci. Found. Ser. Math. Sci.
Springer, Cham, 2022, xii+209 pp.



\bibitem{BHU} P. Bhunia, Improved bounds for the numerical radius via polar decomposition of operators,
Linear Algebra Appl. 683 (2024), 31–45.

\bibitem{BHU-G} P. Bhunia, Improved bounds for the numerical radius via a new norm on $B(H)$,
Georgian Math. J. 32 (2025), no. 4, 551–566.


\bibitem{Bhunia-RMJM} P. Bhunia, S. Jana and K. Paul, Refined inequalities for the numerical radius of Hilbert space operators,
Rocky Mountain J. Math. 55 (2025), no. 2, 323–332.

\bibitem{BKS} P. Bhunia, F. Kittaneh and S. Sahoo,  Improved numerical radius bounds using the Moore-Penrose inverse,
Linear Algebra Appl. 711 (2025), 1-16.


\bibitem{Bhunia-CMB} P. Bhunia, Inequalities of the $\ell^p$-operator norm for block matrices, Canad. Math. Bull. (2026), 12 pp. http://dx.doi.org/10.4153/S0008439526101714

\bibitem{Bourin} J.-C. Bourin and F. Hiai, Jensen and Minkowski inequalities for operator means and anti-norms, Linear Algebra Appl. 456 (2014) 22–53.

\bibitem{Buzano} 
M. L. Buzano, { Generalizzatione della disuguaglianza di Cauchy-Schwarz}, Rend. Semin. Mat. Univ. Politech. Torino 31(1971/73) (1974), 405--409.

\bibitem{DRAG} S. S. Dragomir,
Some inequalities for the norm and the numerical radius of linear operators in Hilbert spaces,
Tamkang J. Math. 39 (2008), no. 1, pp. 1-7.

\bibitem{Fong} C. K. Fong and J. A. R. Holbrook, Unitarily invariant operator norms, Canadian J. Math. 35 (1983), no. 2, 274–299.

\bibitem{Gustafson_Book_1997}
K. E. Gustafson, D. K. M. Rao, {Numerical range}, Springer, New York, 1997.


\bibitem{Jana} S. Jana, P. Bhunia and K. Paul, Euclidean operator radius inequalities of a pair of bounded linear operators and their applications,
Bull. Braz. Math. Soc. (N.S.) 54 (2023), no. 1, Paper No. 1, 14 pp.

\bibitem{Kittaneh88}  F. Kittaneh,  Notes on some inequalities for Hilbert space operators, Publ. Res. Inst. Math. Sci. 24 (1988), no. 2, 283–293.

\bibitem{Kittaneh1} F. Kittaneh, Numerical radius inequality and an estimate for the numerical radius of the Frobenius companion matrix, Studia Math. 158 (2003), no. 1, 11-17.

\bibitem{Kittaneh2} F. Kittaneh, Numerical radius inequalities for Hilbert space operators, Studia Math. 168 (2005), no. 1, 73-80.

\bibitem{Kit23} F. Kittaneh, H. R. Moradi and M. Sababheh, Sharper bounds for the numerical radius, Linear Multilinear Algebra 73 (2025), no. 10, 2309–2319. 


\bibitem{KMS} F. Kittaneh, H. R. Moradi and M. Sababheh,  Complementary bounds for the numerical and spectral radii,
Quaest. Math. 48 (2025), no. 6, 863–877.


\bibitem{Mali} L. Maligranda, Simple norm inequalities, Amer. Math. Monthly 113 (2006), no. 3, 256–260.


\bibitem{MOS} M. S. Moslehian, M. Sattari, K. Shebrawi, Extensions of Euclidean operator radius inequalities. Math. Scand. 120 (2017), no. 1, 129–144.

\bibitem{Popescu_MAMS_2009}
G. Popescu, {Unitary invariants in multivariable operator theory}, Mem. Amer. Math. Soc. 200 (2009), No. 941, vi+91 pp.

\bibitem{SM} M. Sababheh and H. R. Moradi,  New norm and numerical radius bounds, Internat. J. Math. 36 (2025), no. 1, Paper No. 2450066, 17 pp.

\bibitem{SMS} M. Sababheh, H. R. Moradi and S. Sahoo, Inner product inequalities with applications, Linear Multilinear Algebra 73 (2025), no 9, 2089–2102.  



\bibitem{V} M.P. Vasi\'c,  D. J. Ke\^cki\'c, Some inequalities for complex numbers, Math. Balkanica 1 (1971), 282-286.

\bibitem{Wu}  P. Y. Wu and H.-L. Gau. Numerical Ranges of Hilbert Space Operators.  Encyclopedia Math. Appl. 179, Cambridge University Press, Cambridge, 2021.








\end{thebibliography}

\end{document}